\documentclass[11pt,oneside,english]{amsart}
\usepackage{soul} 

\usepackage[T1]{fontenc}
\usepackage[a4paper]{geometry}
\usepackage{color}
\usepackage{amstext}
\usepackage{amsthm}
\usepackage{amssymb}
\usepackage{varioref}
\usepackage{hyperref}
\usepackage{cleveref}
\usepackage{amssymb}
\usepackage{mathrsfs}
\usepackage{enumerate}
\usepackage[all]{xy}
\hypersetup{
    colorlinks,
    citecolor=blue,
    filecolor=black,
    linkcolor=blue,
    urlcolor=black
}
\title{Blown-up  intersection cochains and Deligne's sheaves}

\date{\today}

\author{David Chataur}
\address{Lamfa\\
Universit\'e de Picardie Jules Verne\\
33, rue Saint-Leu\\
80039 Amiens Cedex~1\\
         France}
\email{David.Chataur@u-picardie.fr}

\author{Martintxo Saralegi-Aranguren}
\address{Laboratoire de Math{\'e}matiques de Lens\\  
      EA 2462 \\
      Universit\'e d'Artois\\
         SP18, rue Jean Souvraz\\
          62307 Lens Cedex\\
         France}
\email{martin.saraleguiaranguren@univ-artois.fr}

\author{Daniel Tanr\'e}
\address{D\'epartement de Math{\'e}matiques\\
         UMR 8524 \\
         Universit\'e de Lille\\
         59655 Villeneuve d'Ascq Cedex\\
         France}
\email{Daniel.Tanre@univ-lille.fr}

\thanks{This research was supported through the program ``Research in Pairs'' at the Mathematisches
Forschunginstitut Oberwolfach in 2016. The authors thank the MFO for its generosity and hospitality.
The third author was also supported by the MINECO and FEDER research project MTM2016-78647-P}



\makeatletter
\renewcommand\l@subsection{\@tocline{2}{0pt}{2pc}{5pc}{}}
\renewcommand\l@subsubsection{\@tocline{3}{0pt}{4pc}{10pc}{}}
\makeatother

\theoremstyle{plain}
\newtheorem{theorem}{Theorem}

\newtheorem{propositionn}{Proposition}

\newtheorem{proposition}{Proposition}[section]

\newtheorem{lemma}[proposition]{Lemma}
\newtheorem{corollary}[proposition]{Corollary}

\theoremstyle{definition}
\newtheorem{definition}[proposition]{Definition}
\newtheorem{example}[proposition]{Example}

\theoremstyle{remark}
\newtheorem{remark}[proposition]{Remark}

\numberwithin{equation}{section}
\newcommand{\secref}[1]{Section~\ref{#1}}

\newcommand{\thmref}[1]{Theorem~\ref{#1}}
\newcommand{\propref}[1]{Proposition~\ref{#1}}
\newcommand{\lemref}[1]{Lemma~\ref{#1}}
\newcommand{\corref}[1]{Corollary~\ref{#1}}

\newcommand{\defref}[1]{Definition~\ref{#1}}

            \newcommand{\ov}{   \overline } 
      \def\tc{{\mathtt c}}
            \def\tv{{\mathtt v}}
            \def\tu{{\mathtt u}}
                    \def\gd{{\mathfrak d}}
\def\gC{{\mathfrak C}}
\def\gc{{\mathfrak c}}
\def\gH{{\mathfrak H}}

\def\rc{{\mathring{\tc}}}
\def\rB{{\mathring{B}}}
\def\codim{{\mathrm{codim\,}}}
\def\colim{{\varinjlim\,}}
\def\lim{{\mathrm{lim\,}}}
\def\max{{\mathrm{max\,}}}
\def\pr{{\rm pr\,}}
\def\Ext{{\mathrm{Ext\,}}}
\def\BM{{\mathrm{BM}}}
\def\im{{\mathrm{Im\,}}}
\def\Sd{{\mathrm{Sd}}}
\def\Cov{{\mathrm{Cov}}}

\def\N{{\mathbb N}}
\def\R{{\mathbb R}}

\DeclareMathOperator{\class}{{\rm class}}
\def\Hom{{\mathrm{Hom}}}
\def\Coker{{\mathrm{Coker}}}
\def\Ker{{\mathrm{Ker}}}
\def\reg{{\mathrm{reg}}}
\def\cD{\mathcal{D}}
\def\cP{\mathcal{P}}
\def\tN{\widetilde{N}}

\newcommand{\IH}{\mathscr H}
\def\crH{{\mathscr H}}
\def\id{{\mathrm{id}}}
\def\supp{{\mathrm{Supp}}}

\begin{document}

\begin{abstract}
In a series of papers the authors introduced the so-called  blown-up  intersection cochains. These cochains are suitable
 to study products and cohomology operations of intersection cohomology of stratified spaces. 
The aim of this paper is to prove that the sheaf versions of the functors of  blown-up  intersection cochains 
are realizations of Deligne's sheaves. This proves that Deligne's sheaves can be incarnated 
at the level of complexes of sheaves by soft sheaves of perverse differential graded algebras. 
We also study Poincar\'e and Verdier dualities of  blown-up  intersections sheaves with the use of
Borel-Moore chains of intersection.
\end{abstract}

\maketitle

\tableofcontents

Let $R$ be a commutative unitary ring and let $M$ be an oriented topological manifold of dimension $n$, without
boundary, but not necessarily compact. Poincar\'e duality is satisfied
as an isomorphism 
\[
H^{*}(M;R)\cong H_{n-*}^{\BM}(M;R)
\]
between singular cohomology and Borel-Moore homology (see \cite{BM}). There is also a compact supports version which is an isomorphism
\[
H^{*}_c(M;R)\cong H_{n-*}(M;R)
\]
between singular cohomology with compact supports and singular homology. 
These two isomorphisms come from the cap product with a fundamental class $[M]\in H^{\BM}_n(M;R)$.

\medskip
\paragraph{\bf Poincar\'e duality and sheaf theory}
Let us recall how to interpret these isomorphisms in terms of complexes of sheaves of $R$-modules, 
where $R$ is a principal ideal domain. We consider  two sheaves, $\mathbf{N}^*$ and $\mathbf{C}^*_{\BM}$.
The sheaf of singular cochains $\mathbf{N^{*}}$ is the sheafification of the presheaf of singular cochains
$$U\mapsto N^*(U;R).$$
The sheaf of
Borel-Moore chains $\mathbf{C}^{*}_{\BM}$ is  the sheafification of the presheaf of relative locally finite singular chains (\cite[section 3.1]{MR2276609})
$$U\mapsto C^{\BM}_{n-*}(M,M\backslash\overline{U};R).$$
These two sheaves
are both acyclic
resolutions of the constant sheaf $\mathbf{\underline{R}}$. In the derived category $\mathcal{D}(M)$ of complexes of sheaves of $R$-modules we have the isomorphisms,
$$\underline{\mathbf{R}}\cong\mathbf{N}^*\cong\mathbf{C}^*_{\BM}.$$
Then taking hypercohomology we have the following isomorphisms,
$$H^*(M;R)\cong\mathbb{H}^*(M;\mathbf{\underline{R}})\cong 
\mathbb{H}^*(M;\mathbf{N^*})\cong \mathbb{H}^*(M;\mathbf{C}^*_{\BM})\cong H_{n-*}^{\BM}(M;R),$$
and the corresponding compact supports version,
$$H^*_c(M;R)\cong\mathbb{H}^*_c(M;\mathbf{\underline{R}})\cong 
\mathbb{H}^*_c(M;\mathbf{N^*})\cong \mathbb{H}^*_c(M;\mathbf{C}^*_{\BM})\cong H_{n-*}(M;R).$$
Let $\mathbf{F}^*$ be a complex of sheaves on $M$.
Borel, Moore and Verdier introduced 
a dual sheaf $\mathbb{D}(\mathbf{F}^*)$ (see \cite[V.7]{Bor}) which gives an isomorphism
$$\mathbb{D}(\mathbf{C}^*_{\BM}[n])\cong \mathbf{N}^*$$
in the derived category $\mathcal{D}(M)$. This proves that the constant sheaf $\underline{\mathbf{R}}$ is self-dual
for Verdier duality: 
$$\mathbb{D}(\underline{\mathbf{R}}[n])\cong \underline{\mathbf{R}}.$$
As a consequence, we get
a bilinear pairing
\[
H^{i}(M;R)\otimes H^{n-i}_c(M;R)\rightarrow R
\]
between singular cohomology and singular cohomology with compact supports
that is nonsingular on the torsion free parts of $H^{i}(M;R)$ and $H^{n-i}_c(M;R)$.

\medskip
\paragraph{\bf Poincar\'e duality and singular spaces}
For singular spaces, Poincar\'e duality is not satisfied in general. However when the space
is stratified, Verdier duality still makes sense and we can ask for
the existence of self-dual sheaves. 
\\
In the case of an oriented pseudomanifold $X$ of dimension $n$, Goresky and MacPherson introduced complexes of sheaves called intersection complexes $\mathbf{IC}^*_{\overline{p}}$ parametrized by sequences of integers called perversities $\overline{p}$ (\cite{GM1, GM2}). For a fixed perversity $\overline{p}$ they give a list of axioms, the Deligne's axioms, that characterize 
$\mathbf{IC}^*_{\overline{p}}$ up to isomorphism in the derived category $\mathcal{D}(X)$. 
Working over a field $\mathbb{F}$, they proved the following duality result 
$$\mathbb{D}(\mathbf{IC}^*_{\overline{p}}[n])\cong \mathbf{IC}^*_{D\overline{p}}$$
where $D\overline{p}$ is the complementary perversity of $\overline{p}$. As a consequence, in cohomology, we get a 
nonsingular bilinear pairing
$$\mathbb{H}^i(X;\mathbf{IC}^*_{\overline{p}})\otimes 
\mathbb{H}^{n-i}_c(X;\mathbf{IC}^*_{D\overline{p}})
\rightarrow \mathbb{F}.$$
This pairing comes from an intersection type product
$$\mathbf{IC}^i_{\overline{p}}\otimes\mathbf{IC}^j_{\overline{q}}\rightarrow \mathbf{IC}^{i+j}_{\overline{p}+\overline{q}}$$ 
defined in the derived category $\mathcal{D}(X)$. 

\medskip
\paragraph{\bf Blown-up cochains}
In \cite{CST1,CST6, CST4} we introduced a functor of singular cochains 
$\widetilde{N}^*_{\overline{p}}$, the  blown-up  cochains (also called Thom-Whitney cochains in the first works). 
In fact, we have a family of cochain complexes $\{\widetilde{N}^*_{\overline{p}}(X;R)\}_{\overline{p}\in\mathcal{P}}$
 compatible with the poset of perversities $\mathcal{P}$  that comes equipped with a cup product,
$$\widetilde{N}^i_{\overline{p}}(X;R)\otimes\widetilde{N}^j_{\overline{q}}(X;R)
\xrightarrow{-\cup-}
\widetilde{N}^{i+j}_{\overline{p}+\overline{q}}(X;R),
$$
whose construction is related to the classical cup product on the singular cochain complex $N^*(X;R)$.

\medskip
There is also the singular chain complex $C^{\overline{p}}_*(X;R)$ introduced by King in \cite{MR800845}
whose sheafification is the intersection sheaf of Goresky and MacPherson. The homology of 
$C^{\overline{p}}_*(X;R)$ is called 
intersection homology.
Let us notice that to intersect chains in $C^{\overline{p}}_*(X;R)$ one generally needs to deform them, 
the cup product defined on
$\widetilde{N}^*_{\overline{p}}(X;R)$
obviates this requirement. Moreover, in \cite{CST4} we introduced a cap product between 
$\widetilde{N}^*_{\overline{p}}(X;R)$ and a generalization $\mathfrak{C}^{\overline{p}}_*(X;R)$ of   
$C^{\overline{p}}_*(X;R)$,
better adapted to large perversities, see \defref{tameNormHom}. 
When $X$ is compact and oriented, we proved that the cap product
with a cycle $\gamma_{X}\in \mathfrak{C}^{\overline{0}}_n(X;R)$ representing the fundamental class  
$\Gamma_X\in \mathfrak{H}^{\overline{0}}_n(X;R)$,
$$\widetilde{N}^i_{\overline{p}}(X;R)
\xrightarrow{-\cap\gamma_X}
\mathfrak{C}^{\overline{p}}_{n-i}(X;R),$$
is a quasi-isomorphism for any ring $R$ (\cite[Theorem B]{CST2}). 
This extends the Poincar\'e duality theorem of Friedman and McClure \cite{FM} established for field coefficients.

\medskip
\paragraph{\bf Outline.}
In this paper we study the sheafification $\mathbf{\widetilde{N}}^*_{\overline{p}}$ (see \defref{def:covV})
of the  blown-up  cochains $\tN^*_{\ov{p}}$. 
After a brief recall in \secref{sec:recall} of the basic definitions, Theorem A of \secref{sec:boum} is our first result:
\emph{ the sheaf $\mathbf{\widetilde{N}}^*_{\overline{p}}$ satisfies Deligne's axioms.} 
As a consequence, we get that $\mathbf{\widetilde{N}}^*_{\overline{p}}$ offers a way to promote 
the multiplicative structure of Deligne's sheaves directly on the complex of sections of 
$\mathbf{\widetilde{N}}^*_{\overline{p}}$.  This multiplicative enrichment at the level of sections has some nice 
consequences for the study of intersection cohomology:   it makes the development 
of a rational homotopy theory possible in this context (\cite{CST1}) and brings a better paradigm for the study of 
Steenrod squares in intersection cohomology (\cite{CST6}).

Let $R$ be a principal ideal domain and $I^*_{R}$ an injective resolution of $R$.
In \secref{sec:BMchains}, we introduce  a complex
of Borel-Moore chains of intersection, denoted by $\mathfrak{C}^{\BM,\overline{p}}_{*}(X;R)$.
 In the case of a PL-pseudomanifold,
we prove (see \propref{prop:spanier}):

\begin{propositionn}\label{prop:first}
Let $X$ be a paracompact PL-pseudomanifold with a perversity $\ov{p}$. Then, there exists a 
quasi-isomorphism, linear for the  structure of perverse
$\widetilde{N}_{\overline{q}}^{*}(X;R)$-modules,
$$
\Phi_X\colon \mathfrak{C}^{\BM,\overline{p}}_{k}(X;R)\rightarrow 
{\Hom}_{k}(\mathfrak{C}^*_{\overline{p},c}(X;R),I_R^*).
$$
\end{propositionn}
Thus
${\Hom}_{k}(\mathfrak{C}^*_{\overline{p},c}(X;R),I_R^*)$
 appears as the bidual of the complex of intersection chains.

\medskip
In \secref{sec:verdier}, we build a diagram of cochain complexes,
$$\xymatrix{
\widetilde{N}^k_{\overline{p}}(X;R)
\ar[r]^-{\mathcal{D}_{X}}
\ar[rd]_-{DP_X}
&{\Hom}_{n-k}(\mathfrak{C}^{*}_{\overline{p},c}(X;R),I_R^*)
\\
&
\mathfrak{C}_{n-k}^{\BM,\overline{p}}(X;R),
\ar[u]_-{\Phi_X}
}$$
where \begin{itemize}
\item $DP_{X}$ is the cap product with a fundamental cycle 
$\gamma_X\in \mathfrak{C}^{\BM,\overline{0}}_{n}(X;R)$,
\item $\Phi_X$ is the morphism of \propref{prop:first}, 
\item $\mathcal{D}_{X}$ is a product with the cocycle $\Phi_{X}(\gamma_{X})$.
\end{itemize}
This diagram relates directly at the chain level Poincar\'e and Verdier dualities
and we prove (see \propref{prop:module} and \corref{cor:poincareBM}):

\begin{propositionn}
The previous diagram is commutative  and all maps are linear quasi-isomorphisms 
for the  structure of perverse
$\widetilde{N}_{\overline{q}}^{*}(X;R)$-modules.
\end{propositionn}

\medskip
We fix for the sequel a \textbf{principal ideal domain} $R$.  For a topological space $X$, we denote by  $\rc X = X \times [0,1[/ X \times \{ 0\}$ the \textbf{open cone} on $X$. The apex of the cone is $\tv =[-,0]$.

\section{Pseudomanifolds and perversities}\label{sec:recall}
\begin{quote}
This section is a recall of definitions of  basic objects: pseudomanifolds and perversities. We consider perversities defined on each stratum
which cover the original case of Goresky and MacPherson.
\end{quote}

\subsection{Topological background}
\begin{definition}
A \textbf{filtered space} $(X,(X_{i})_{0\leq i\leq n})$,
is a  Haussdorff  topological space $X$
equipped with a filtration
by closed subspaces, 
\[
X_{0}\subseteq X_{1}\subseteq\dots\subseteq X_{n}=X,
\]
such that the complement, $X_{n}\backslash X_{n-1}$, is non-empty.
The integer $d(X)=n$ is the \textbf{formal dimension} of
$X$. The connected components, $S$, of $X_{i}\backslash X_{i-1}$
are the \textbf{strata} of $X$ and we set $d(S)=i$. The strata of
$X_{n}\backslash X_{n-1}$ are called\textbf{ regular}, the other
being called singular. We denote by $\mathcal{S}_{X}$ the set of
non-empty strata of $X$. The sub-space $X_{n-1}$ is the\textbf{
singular set} also  denoted by $\Sigma _X$.
\end{definition}

\begin{definition}
A \textbf{stratified topological pseudomanifold} or pseudomanifold for short, of dimension $n$ is a filtered topological space
 $(X,(X_{i})_{0\leq i\leq n})$
such that, for any $i$, $X_{i}\backslash X_{i-1}$ is a
topological manifold of dimension $i$ or the empty set. 
Moreover, for any point $x\in X_{i}\backslash X_{i-1}$,$i\neq n$,
we have:
\begin{enumerate}
\item an open neighbourhood, $V$, of $x\in X$, together with the \textbf{induced
filtration } $V_i = V \cap X_i$, 
\item an open neighbourhood, $U$, of $x\in X_{i}\backslash X_{i-1}$,
\item a compact pseudomanifold $L$ of dimension $n-i-1$,
called the \textbf{link} of $x$, the cone $\overset{\circ}{c} L$ being endowed
with the \textbf{conical filtration } $(\overset{\circ}{c} L)_{i}=\overset{\circ}{c}(L_{i-1})$,  
\item a homeomorphism $\varphi\colon U\times  \overset{\circ}{c} L\to V$ such that:\end{enumerate}
\begin{itemize}
\item $\varphi(u,\tv)=u$, for all $u\in U$, 
\item $\varphi(U\times \overset{\circ}{c} L_{j})=V\cap X_{i+j+1}$, for all $j\in\{0,\dots,n-i-1\}$.
\end{itemize} 
The pair $(\varphi, V)$ is called a \textbf{local chart.}
A pseudomanifold $X$ is {\bf $R$-oriented} if its regular part $X\backslash \Sigma$ is $R$-oriented.
\end{definition}

\begin{definition}
Let $(X,(X_{i})_{0\leq i\leq n})$ and $(Y,(Y_{i})_{0\leq i\leq n})$ be pseudomanifolds.
A \textbf{stratum preserving map} is a continuous
map $f\colon X\rightarrow Y$ such that for any stratum $S\in\mathcal{S}_{X}$
there exists a stratum $S^{f}\in\mathcal{S}_{Y}$ such that $f(S)\subset S^{f}$
and $\codim(S^{f})= \codim(S)$.
\end{definition}

\subsection{Perversities and perverse spaces}
\begin{definition}\label{def:perversity}
A \textbf{perversity} on a filtered topological space $(X,(X_{i})_{0\leq i\leq n})$,
is a function, $\overline{p}\colon\mathcal{S}_{X}\to\mathbb{Z}$,
defined on the set of strata of $X$ and such that $\overline{p}(S)=0$
on any regular stratum $S$. The couple $(X,\overline{p})$ is called
a \textbf{perverse space}. If $X$ is a pseudomanifold, we say a \textbf{perverse pseudomanifold.}
We denote by $\mathcal{P}$ the set of perversities on $X$.  
\end{definition}

\noindent Let us review some important perversities.
\begin{enumerate}
\item A \textbf{GM perversity} (for Goresky and Mac-Pherson, see \cite{GM1}) is a sequence of integers
$\overline{p}\colon \{0,\dots,n\}\rightarrow\mathbb{N}$ verifying 
$\overline{p}(0)=\overline{p}(1)=\overline{p}(2)=0$ and $\overline{p}(i)\leq\overline{p}(i+1)\leq\overline{p}(i)+1$.
It induces a perversity as above, defined as $\overline{p}(S)=\overline{p}(\codim(S))$.
\item The \textbf{zero perversity} $\overline{0}$ is the GM perversity defined as $\overline{0}(i)=0$
for any $i \in \{0,\dots,n\}$.

\item The \textbf{top perversity} $\overline{t}$ is the GM perversity given by $\overline{t}(i)=i-2$
for $i\geq2$. We say that two perversities $\overline{p}$ and $\overline{q}$
are \textbf{complementary} if for any singular stratum $S$ we have:
\[
\overline{p}(S)+\overline{q}(S)=\overline{t}(S)=n-\dim(S)-2.
\]
Each perversity $\overline{p}$ has a {\bf complementary perversity} 
$D\overline{p}$ defined by $D\overline{p}=\overline{t}-\overline{p}$. When $\overline{p}$ is a GM perversity $D\overline{p}$ is also GM.
\end{enumerate}

\begin{definition}
Let $f\colon X\rightarrow Y$ be a stratum preserving map and  $\overline{q}$
 a perversity on $Y$. We define the\textbf{ pull-back perversity}
$f^{*}\overline{q}$ thanks to the formula:
\[
f^{*}\overline{q}(S)=\overline{q}(S^{f})
\]
for any $S\in S_{X}$. 
\end{definition}

\begin{example}
Let us review the three fundamental examples that we encounter
in this paper. 

\begin{enumerate}
\item Let $U\subset X$
be an open set of a pseudomanifold, endowed with the induced filtration. The inclusion $i \colon U\hookrightarrow X$
is a stratum preserving map. By abuse of notation, the pull-back perversity  $i^{*}\overline{p}$ is also  denoted by $\overline{p}$.

\item Let $M$ be an $m$-dimensional manifold and $(X,\ov{p})$ be a pseudomanifold. The \textbf{product filtration }on $X\times M$ is defined by
$(X\times M)_{i}=X_{i-m}\times M$. The projection map 
$
\pi\colon X\times M\rightarrow X
$
is a stratum preserving morphism. By abuse of notation, the pull-back perversity $\pi^{*}\overline{p}$
will also be denoted by $\overline{p}$.

\item Let $L$ be a compact filtered space. A perversity $\overline p$ on the cone $\rc L$ is determinate by  the number $\overline p (\tv)$
and a perversity on $L$, still denoted by $\overline p$. They are related by the formula $\overline p(S) = \overline p(S \times ]0,1[)$, for each $S \in \mathcal S_L$.
\end{enumerate}
\end{example}

\subsection{Perverse structures}

Unlike singular cohomology that comes equipped with a cup product, intersection cohomology at a fixed perversity 
is not an algebra. Rather it is parametrized by a family of perversities, and it has
a product that is compatible with perversities. This has been formalized
by M. Hovey who introduced perverse chain complexes, perverse differential
graded algebras in \cite{Hov}. Let us review this formalism. It relies
on the fact that the set of perversities has two remarkable features.

(1 ) The set $\mathcal{P}$ of perversities on a filtered space $X$ is a poset: we say that
$\overline{p}\geq\overline{q}$ if for any stratum $S$ we have $\overline{p}(S)\geq\overline{q}(S)$. 

(2) We can add perversities: $(\overline{p}+\overline{q})(S)=\overline{p}(S)+\overline{q}(S)$.
And the zero perversity $\overline{0}$ is a unit for this addition.

\smallskip
Thanks to these two properties we define perverse chain complexes,
perverse algebras. 
\begin{definition}
Let $(\mathcal{M},\square,\mathbb{I})$ be a symmetric monoidal category. 

(1) A \textbf{perverse object} in $\mathcal{M}$ is a functor $M_{\bullet}\colon \mathcal{P}\rightarrow \mathcal{M}$.
That is,  we have for each perversity $\overline{p}$ an object $M_{\overline{p}}$
of $\mathcal{M}$ and a collection of morphisms of $\mathcal{M}$,
\[
\gamma_{\overline{p}\leq\overline{q}}\colon M_{\overline{p}}\rightarrow M_{\overline{q}},
\]
whenever $\overline{p}\leq\overline{q}$ such that $\gamma_{\overline{p}\leq\overline{p}}=\id$
and $\gamma_{\overline{q}\leq\overline{r}}\circ\gamma_{\overline{p}\leq\overline{q}}=\gamma_{\overline{p}\leq\overline{r}}$. 

(2) A \textbf{perverse monoid} in $\mathcal{M}$ is a perverse object $M_{\bullet}$
together with a unit $\eta\colon \mathbb{I}\rightarrow M_{\overline{0}}$ and associative products,
\[
\mu\colon M_{\overline{p}}\square M_{\overline{q}}\rightarrow M_{\overline{p}+\overline{q}},
\]
 compatible with the poset structure.
 A perverse monoid in the category of complexes is called a \textbf{perverse differential graded algebra}.

(3) A (left) \textbf{perverse module} over a perverse monoid $A_{\bullet}$ 
is a perverse object $M_{\bullet}$ together with products
$$\bot\colon A_{\ov{p}}\square M_{\ov{q}}\to M_{\ov{p}+\ov{q}}$$
compatible with the poset structure and satisfying the classical properties of modules. 
\end{definition}

\begin{example}
We will consider the next examples of categories $\mathcal{M}$. 
\\
- The category of graded $R$-modules: intersection homology and cohomogy provide examples of perverse objects in this category. 
\\
- The category of complexes of $R$-modules: the  blown-up  cochains introduced in the next section, equipped with its cup product is a perverse monoid.
\\
- The category of complexes of sheaves of $R$-modules on $X$.
\end{example}

\section{Blown-up  intersection cohomology}\label{sec:boum}

\begin{quote}
In this section we introduce the cochain complex of  blown-up intersection  cochains
$\widetilde{N}_{\overline{p}}^{*}$ and prove,
in \thmref{DeligneTW},  that its sheafification $\mathbf{\widetilde{N}^*_{\overline{p}}}$ 
gives a nice realization of the  Deligne sheaf.
\end{quote}

\subsection{Blown-up  intersection cochains}

We start with the standard $k$-simplex viewed as an ordered simplicial complex $ \Delta$. We associate
to this $k$-simplex the simplicial chain complex $N_{*}( \Delta )$,
whose component of $l$-chains is the $R$-module generated by the $l$-faces of
$ \Delta $, together with the usual differential given by the face
operators. We define $N^{*}( \Delta  )$ as the $R$-linear dual
of $N_{*}( \Delta  )$. This simplicial complex has a geometric
realization which is the standard $k$-simplex embedded in $\mathbb{R}^{k+1}$,
by abuse of notations we also denote this geometric realization as $ \Delta $.

Our constructions use simplices in a "good" position with respect to the filtration of $X$, let us give a precise definition.

\begin{definition}\label{def:filteredsimp}
Let $(X,\overline{p})$ be a perverse space.

(i) A \textbf{filtered simplex} $\sigma \colon \Delta \rightarrow X$
is a singular simplex such that each $\sigma^{-1}(X_{i})$ is a face
of $ \Delta $. A filtered simplex admits a join decomposition:
\[
 \Delta =\Delta_{0}*\dots*\Delta_{n}.
\]
A filtered simplex is said to be \textbf{regular} if $\Delta_{n}\neq\emptyset$.

(ii) Let $FSimp(X)$ be the category whose objects are the 
regular  simplices of $X$ and whose morphisms are induced
by the face operators of the simplices.
\end{definition}

If $\sigma \colon\Delta\rightarrow X$
is a regular simplex and  $F$  a regular face
of $\Delta$, we notice that the canonical map $i\colon F \rightarrow\Delta$
induces a morphism in $FSimp(X)$ between $\sigma\circ i$ and $\sigma$. 

\medskip
Having defined filtered simplices we can introduce the functor of
 blown-up  cochains as a system of coefficients on the category
$FSimp(X)$.

\begin{definition}
Let $(X,\overline{p})$ be a perverse  space and $\sigma \colon\Delta_{0}*\dots*\Delta_{n}\rightarrow X$ a fixed regular simplex.
\\
(i) To $\sigma$ we associate the cochain complex:
\[
\widetilde{N}^{*}_{\sigma}=N^{*}(\tc\Delta_{0})\otimes N^{*}(\tc\Delta_{1})\otimes\dots\otimes N^{*}(\tc\Delta_{n-1})\otimes N^{*}(\Delta_{n}).
\]

\noindent
(ii) For any regular face $F$ of $\sigma$ we have the restriction
morphism:
\[
i^*\colon \widetilde{N}^{*}_{\sigma}\rightarrow\widetilde{N}^{*}_{\sigma\circ i}.
\]
We define the \textbf{blown-up  cochains} as the inverse limit 
\[
\widetilde{N}^{*}(X)={\varprojlim}_{\sigma\in FSimp(X)}\widetilde{N}^{*}_{\sigma}.
\]
\end{definition}

\noindent Said differently a  blown-up  cochain is a collection 
$\left\{ \omega_{\sigma}\right\} _{\sigma \colon\Delta\rightarrow X}$
of tensor products of cochains parametrized by regular simplices,
satisfying gluing conditions with respect to restrictions
to the regular faces. 
 A blown-up cochain owns a perverse degree relatively to any stratum of $X$. Let us see that. 

\begin{definition}
Let  $\Delta = \Delta_0 * \dots * \Delta_n$ be a regular simplex.
For any $1\leq k \leq n$ we consider the restriction map,\\
$\rho_{k}\colon N^{*}(\tc
\Delta_{0})\otimes\dots\otimes N^{*}(\tc\Delta_{n-1})\otimes N^{*}(\Delta_{n})
\to
N^*(\tc\Delta_0)\otimes \dots  \otimes N^*(\tc \Delta_{n-k-1})\otimes N^*(\Delta_{n-k})\otimes N^*(\tc\Delta_{n-k+1})
\otimes \dots\otimes N^{*}(\tc\Delta_{n-1}) \otimes N^{*}(\Delta_{n})
$.\\
For each $\eta \in N^{*}(\tc\Delta_{0})\otimes\dots\otimes N^{*}(\tc\Delta_{n-1})\otimes N^{*}(\Delta_{n})
 $, the cochain $\rho_k(\eta)$ can be written as a finite sum  $\sum_{i\in I} a_i\otimes b_i$,
  where $\{a_i \}_{i\in I}$ is an $R$-basis of the free module $N^*(\tc\Delta_0)\otimes \dots \otimes N^*(\Delta_{n-k} )$
   and $b_i\in N^*(\tc\Delta_{n-k+1})\otimes\dots\otimes N^*(\Delta_n)$. We set 
\[
\left\Vert \eta\right\Vert _{k}=\begin{cases}
- \infty & \text{if}\;\rho_k(\eta)=0,\\
\max\{\deg(b_i)\mid b_i\neq 0\} & \text{if}\;\rho_k(\eta)\neq 0.
\end{cases}
\]
\end{definition}
\begin{definition}
Let $\omega$ be a cochain of $\widetilde N^*(X;R)$.
The \textbf{perverse degree of $\omega$ along a singular stratum,} $S \in \mathcal{S}_{X}$, is equal to
\begin{equation*}\label{equa:perversstrate}
\|\omega\|_{S}=\sup\left\{\|\omega_{\sigma}\|_{\codim S}\mid 
\sigma\colon\Delta\to X \text{ regular with }\sigma(\Delta)\cap S\neq \emptyset\right\}.
\end{equation*}
where $\sup \emptyset=-\infty$.

Let us fix a perversity $\overline{p}$ on $X$. We say that $\omega\in \widetilde N^*(X;R)$
is \textbf{$\overline{p}$-allowable} if for any singular stratum
$S$ we have $\left\Vert \omega\right\Vert _{S}\leq\overline{p}(S)$.
The  cochain $\omega$ is a \textbf{$\ov{p}$-intersection cochain} if $\omega$ and its coboundary, $\delta \omega$,
are $\ov{p}$-allowable.  We denote by $\widetilde{N}^*_{\ov{p}}(X;R)$
the complex of  $\overline{p}$-intersection cochains and
by $\IH^*_{\overline{p}}(X;R)$  its homology, called
 \textbf{blown-up  $\ov p$-intersection cohomology} of $X$ with coefficients in~$R$,
for the perversity $\overline{p}$.
\end{definition}

\subsection{Properties of  blown-up  cochains}

Let us recall the main properties of  blown-up  cochains (see \cite{CST4}, \cite{CST2} 
or \cite{CST1}, \cite{CST6}, \cite{CST7} where they are called TW-cochains). 
They satisfy 
the Mayer-Vietoris property \cite[Theorem C]{CST4} and are natural with respect to stratum preserving morphisms (for a careful study of this naturality \cite[Theorem A]{CST4}).

We also have a \textbf{stratified homotopy invariance} \cite[Theorem D]{CST4}.
 If $(X,\overline p)$ is a perverse space then
the canonical projection $\pr \colon X\times\mathbb{R}\rightarrow X$, which is a stratum preserving map, induces 
a quasi-isomorphism
\begin{equation}\label{pro}
\pr^* \colon\widetilde{N}_{\overline{p}}^{*}(X;R)\rightarrow\widetilde{N}_{\overline{p}}^{*}(X\times\mathbb{R};R).
\end{equation}

A crucial property is the \textbf{computation of the blown-up intersection cohomology of an open cone}.
 Let $L$ be a compact filtered topological
space and  $\overline p$  a perversity on the cone $\rc L$, of apex $\tv$. We have  \cite[Theorem E]{CST4}:
\begin{equation}\label{ConeTW}
\IH_{\overline{p}}^{k}(\overset{\circ}{c}L;R)=\begin{cases}
\IH_{\overline{p}}^{k}(L;R), & \text{if }k\leq\overline{p}(\tv),\\
0 & \text{if }k>\overline{p}(\tv).
\end{cases}
\end{equation}
If $k\leq\overline{p}(\tv)$, the previous isomorphism is given by the inclusion $L \times ]0,1[ = \rc L - \{\tv\} \hookrightarrow \rc L$.

\medskip
Moreover, the collection $\left\{ \widetilde{N}_{\overline{p}}^{*}(X;R)\right\} _{\overline{p}\in \mathcal{P}}$
is a perverse differential graded algebra whose products,
$$
-\cup-\colon
\widetilde{N}_{\overline{p}}^{*}(X;R)\otimes\widetilde{N}_{\overline{q}}^{*}(X;R) \longrightarrow
\widetilde{N}_{\overline{p}+\overline{q}}^{*}(X;R),
$$
are induced by the classical cup product on the tensor product 
$N^*(\tc\Delta_0)\otimes\dots\otimes N^*(\tc\Delta_{n-1})\otimes N^*(\Delta_n)$ (see \cite[Section 4]{CST4}).\

\subsection{The sheaf of  blown-up intersection cochains}

\begin{definition}\label{def:covV}
Let us consider an open covering $\mathcal{U}$ of $X$ the complex
of \textbf{small $\mathcal{U}$-cochains} denoted by $\widetilde{N}_{\overline{p}}^{*,\mathcal{U}}(X;R)$
is the chain complex of  blown-up  cochains defined on filtered
simplices included in an open set of the covering $\mathcal{U}$, see \cite[Definition 9.6]{CST4}.
If $\mathcal{U}'$ is a covering that refines $\mathcal{U}$, we have
canonical restriction maps 
\[
\widetilde{N}_{\overline{p}}^{*}(X;R)\rightarrow\widetilde{N}_{\overline{p}}^{*,\mathcal{U}}(X;R)\rightarrow\widetilde{N}_{\overline{p}}^{*,\mathcal{U}'}(X.R).
\]
Let $V$ be an open subset of $X$. We define the category $\Cov(V)$
of open coverings of $V$ as the category associated to the poset
of open coverings. The \textbf{sheaf of  blown-up  intersection
cochains} $\widetilde{\mathbf{N}}_{\overline{p}}^{*}$ is defined
by its sections as the direct limit 
\[
\Gamma(V;\widetilde{\mathbf{N}}_{\overline{p}}^{*})=\colim_{\mathcal{U}\in \Cov(V)}\widetilde{N}_{\overline{p}}^{*,\mathcal{U}}(V;R).
\]
\end{definition}

\noindent As the cup product is natural with restriction maps, the family $\{\widetilde{\mathbf{N}}_{\overline{p}}^{*}\}_{\overline{p}\in\mathcal{P}}$
is a  perverse differential graded algebra.

\begin{proposition}\label{SoftFlatTW}
 Let $(X,\overline{p})$ be a perverse space. 
The complex of sheaves $\widetilde{\mathbf{N}}_{\overline{p}}^{*}$  over $X$ 
is a complex of flat and soft sheaves.
\end{proposition}

\begin{proof}
(Flatness) Let us prove that for any $x\in X$ the stalk $(\widetilde{\mathbf{N}}_{\overline{p}}^{*})_x$ is flat.  
A direct limit of flat modules being also flat, it suffices to prove that the module of sections 
$\Gamma(V;\widetilde{\mathbf{N}}_{\overline{p}}^{*})$ is a flat module for any open subset $V\subset X$. 
This module of sections is also the direct limit of small $\mathcal{U}$-cochains 
$\widetilde{N}_{\overline{p}}^{*,\mathcal{U}}(V;R)$.
We are therefore reduced to prove that these small cochains are flat.  
As we work over a principal ideal domain, a torsion free module is flat and a sub-module of a torsion free module 
is also torsion free. Thus let us use the fact that 
$\widetilde{N}_{\overline{p}}^{*}(V;R)\subset \widetilde{N}^{*}(V;R)$ and that 
$\widetilde{N}^{*}(V;R)$ is $R$-torsion free.   
\\
(Softness) For any perversity $\overline{p}$ the complex of sheaves $\widetilde{\mathbf{N}}_{\overline{p}}^{*}$
is a complex of $\widetilde{\mathbf{N}}_{\overline{0}}^{*}$-modules,
this structure being induced by the cup product 
$
\mathbf{\widetilde{N}}_{\overline{0}}^{*}\otimes\widetilde{\mathbf{N}}_{\overline{p}}^{*}\xrightarrow{- \cup -}\widetilde{\mathbf{N}}_{\overline{p}}^{*}.
$
In particular they are $\widetilde{\mathbf{N}}_{\overline{0}}^{0}$-modules.

We have constructed in \cite[Lemma 10.2]{CST4} an $R$-linear operator $N^* (X;R) \to \widetilde N^*_{\overline 0} (X;R)$.
As this  operator preserves the product, we have a morphism of sheaves of rings 
$\mathbf{N}^0\rightarrow \widetilde{\mathbf{N}}_{\overline{0}}^{0}$
between the singular zero cochains and the blown-up  zero cochains in perversity $\ov{0}$.
We recall that a complex of sheaves of modules over a soft sheaf of unital rings is soft. We conclude by using the fact that the sheaf $\mathbf{N}^0$ is soft.
\end{proof}

\begin{proposition}
Let $(X,\overline{p})$  be a perverse space. We  have the following isomorphisms
\[
{\IH_{\overline{p}}^{*}(X;R)\cong H^{*}(\Gamma(X;\widetilde{\mathbf{N}}_{\overline{p}}^{*}))\cong\mathbb{H}^{*}(X;\widetilde{\mathbf{N}}_{\overline{p}}^{*})}.
\]
\end{proposition}

\begin{proof}
The theorem of small $\mathcal{U}$-cochains (see \cite[Corollary 9.8]{CST4}) implies that for any
open covering of $X$ the restriction map $\widetilde{N}_{\overline{p}}^{*}(X;R)\rightarrow\widetilde{N}_{\overline{p}}^{*,\mathcal{U}}(X;R)$
is a quasi-isomorphism. This gives us the first isomorphism, the second
being a direct consequence of the softness of the sheaf $\widetilde{\mathbf{N}}_{\overline{p}}^{*}$.
\end{proof}

\subsection{Blown-up  intersection cochains with compact supports}
\begin{definition}
Let $(X,\overline{p})$  be a perverse space. 
The  \textbf{blown-up  complex of cochains with compact supports}
is the direct limit,
\[
\widetilde{N}_{ \overline{p},c}^{*}(X)=\colim_{K}\widetilde{N}_{\overline{p}}^{*}(X,X\backslash K;R),
\]
where the limit is taken over all compact subsets $K\subset X$. Its cohomology is denoted by
$\IH^*_{\overline{p},c}(X;R)$.
\end{definition}

\begin{proposition}
Let $(X,\overline{p})$  be a perverse space. We  have the following isomorphisms
\[
{ \IH_{\overline{p},c}  ^{*}(X;R)\cong H^{*}(\Gamma_c(X;\widetilde{\mathbf{N}}_{\overline{p}}^{*}))\cong\mathbb{H}_c^{*}(X;\widetilde{\mathbf{N}}_{\overline{p}}^{*}).}
\]
\end{proposition}

\begin{proof}
The second isomorphism follows from the softness of the sheaf $\widetilde{\mathbf{N}}^*_{\overline{p}}$.
In order to get the first one we use the fact that
$$\Gamma_c(X;\widetilde{\mathbf{N}}^*_{\overline{p}})=\colim_{\mathcal{U}\in \Cov(X)} 
\widetilde{N}_{ \overline{p},c}^{*,\mathcal{U}}(X;R)$$
where $\widetilde{N}_{ \overline{p},c}^{*,\mathcal{U}}(X;R)$ is the cochain complex of $\mathcal{U}$-small  blown-up  cochains with compact supports. And to conclude we use the theorem of small $\mathcal{U}$-cochains with compact supports proved in \cite[Proposition 2.5]{CST2}.  
\end{proof}

\subsection{Deligne's sheaves and  blown-up  cochains}

Intersection homology was first  introduced in \cite{GM1} for PL-pseudomanifolds. In \cite{GM2},
Goresky and MacPherson replace the simplicial point of view of \cite{GM1} by a sheaf construction due to 
Deligne \cite{Del}. 
Topological pseudomanifolds and GM perversities
are the paradigm of \cite{GM2}, the main tool being an axiomatic characterization of the Deligne sheaf in the derived category $\mathcal{D}(X)$
of complexes of sheaves on $X$. 
If one replaces GM perversities by the perversities of \defref{def:perversity},
an adaptation $\mathbf{Q}^*_{\overline{p}}$ of the Deligne's sheaf  
 is done by G. Friedman in \cite{FR2}, thanks to
 a generalized truncation functor for sheaves.
 
 \medskip
 Let $(X,\ov{p})$ be a perverse pseudomanifold.
We show that the sheaf $\widetilde{\mathbf{N}}_{\overline{p}}^{*}$
is isomorphic to $\mathbf{Q}^*_{\overline{p}}$ in $\mathcal{D}(X)$.
Let us first recall the Friedman version of the Deligne's axioms for perverse spaces (see \cite[Section 3]{FR2}).
\begin{definition}\label{def:deligne}
Let  $\mathbf{F}^{*}$ be a sheaf complex
on $X$. We denote by $\mathbf{F}_{k}^{*}$ the restriction of $\mathbf{F}^{*}$ to $X\backslash X_{n-k}$.
We say that $\mathbf{F}^{*}$ satisfies the axioms $(AX)_{\overline{p}}$ if:   

\begin{itemize}
\item[(1)] $\mathbf{F}^{*}$ is bounded, $\mathbf{F}^{i}=0$ for $i<0$ and
the restriction $\mathbf{F}_{1}^{*}$ to the regular strata is
quasi-isomorphic to the ordinary singular cochain complex.
\item[(2)] For any stratum $S$ and any $x\in S$, the cohomology sheaf $\mathbf{H}^{i}(\mathbf{F}^{*})_{x}$  vanishes 
if $i>\overline{p}(S)$.

\item[(3)] For any stratum  $S \subset X_{n-k}$ and any $x\in S$,  the attachment
map, $\alpha_{k}\colon \mathbf{F}_{k+1}^{*}\rightarrow Ri_{k}^{*}\mathbf{F}_{k}^{*}$,
induced by the canonical inclusion $X\backslash X_{n-k} \hookrightarrow  X\backslash X_{n-k-1}$
is a quasi-isomorphism at $x$ up to $\overline{p}(S)$. 
\end{itemize}

If $\mathbf{F}^{*}$ is soft, from \cite[Remark 2.3]{Pol}, we may replace the axiom (3)
by the following one.

\begin{itemize}
\item[(3 bis)]  For any  $k\in\{2,\dots,n\}$, stratum $S \subset X_{n-k}\backslash X_{n-k-1}$, $x\in S$ and $j\leq \ov{p}(S)$,
the restriction map induces an isomorphism, 
\[
\colim_{V_{x}}H^{j}(\Gamma(V_{x};\mathbf{F}^{*}))\rightarrow 
\colim_{V_{x}}H^{j}(\Gamma(V_{x}\backslash X_{n-k};\mathbf{F}^{*})),
\]
where $V_{x}$ varies into a cofinal
family of neighborhoods of $x$ in $X\backslash X_{n-k-1}$.
\end{itemize}
\end{definition}

\begin{proposition}{\cite[Proposition 3.8]{FR2}}\label{prop:superdeligne}
Let $(X,\ov{p})$ be a perverse pseudomanifold. Then,
any sheaf complex satisfying the axioms $(AX)_{\ov{p}}$
is isomorphic to $\mathbf{Q}^*_{\overline{p}}$ in $\mathcal{D}(X)$.
\end{proposition}

In \cite{FR1},  G. Friedman extends such unicity theorem  to the  setting of the  
homotopically stratified spaces of F. Quinn, \cite{Qui}. We do not consider these spaces in this work.

\medskip
Our first result relates Deligne's sheaf to the blown-up intersection cohomology.
 
\begin{theorem}\label{DeligneTW}
Let $(X,\overline{p})$ be a perverse pseudomanifold. Then the sheaf complex $\widetilde{\mathbf{N}}_{\overline{p}}^{*}$
is isomorphic to $\mathbf{Q}^*_{\overline{p}}$ in the derived category $\mathcal{D}(X)$ 
of sheaves on $X$. Hence we have the following isomorphism:
\[
{\IH_{\overline{p}}^{*}(X;R)\cong\mathbb{H}^{*}(X;\mathbf{Q}_{\overline{p}}).}
\]
Moreover, this  isomorphism 
is compatible with the products. 
\end{theorem}

\begin{proof}
With \propref{prop:superdeligne}, it is sufficient to prove  that the sheaf $\widetilde{\mathbf{N}}_{\overline{p}}^{*}$
satisfies the axioms $(AX)_{\overline{p}}$.

 Property  (1) follows from the fact that the restriction
of the sheaf $\widetilde{\mathbf{N}}_{\overline{p}}^{*}$ to the
regular strata is the sheafification of the singular cochains.

The fact that $\widetilde{\mathbf{N}}_{\overline{p}}^{*}$ satisfies
(2) and  (3 bis)  follows from
the softness of this sheaf  and the properties of  blown-up  intersection
cohomology recalled in (\ref{ConeTW}). Let us begin with axiom $(2)$. 
Let   $S$ be a stratum of $X_{n-k}$ and $x\in S$. 
 We choose a cofinal family of local charts
 such that $\varphi\colon V_{x} \xrightarrow{\cong}\mathbb{R}^{n-k}\times \rc L$ where
$L$ is the link of $x$.  
By softness of $\widetilde{\mathbf{N}}_{\overline{p}}^{*}$
we have the following isomorphisms:
\[
\mathbf{H}^{i}(\widetilde{\mathbf{N}}_{\overline{p}}^{*})_{x}\cong \colim_{V_{x}}H^{i}(\Gamma(V_{x},\widetilde{\mathbf{N}}_{\overline{p}}^{*}))\cong \colim_{V_{x}}\IH_{\overline{p}}^{i}(V_{x},R).
\]
Following \eqref{pro} we get that
\[
\mathbf{H}^{i}(\widetilde{\mathbf{N}}_{\overline{p}}^{*})_{x}
\cong 
\crH_{\overline{p}}^{i}(\mathbb{R}^{n-k}\times \rc L;R)
\cong 
\crH_{\overline{p}}^{i}(\rc L;R).
\]
Then $(2)$ follows from the fact that $\crH_{\overline{p}}^{i}(\rc L;R)=0$
whenever $i>\overline{p}(S)$, cf. \eqref{ConeTW}. In order to prove that
$\widetilde{\mathbf{N}}_{\overline{p}}^{*}$ satisfies  (3 bis), we use
 the same cofinal family of local charts. 
 We have to analyze the attaching map 
 $
\crH_{\overline{p}}^{i}(V_{x};R)
\rightarrow 
\crH_{\overline{p}}^{i}(V_{x}\backslash X_{n-k};R)$. 
Using $\varphi$, the attaching map becomes 
$
\crH_{\overline{p}}^{i}(\mathbb{R}^{n-k}\times \rc L;R)  
\rightarrow  
\crH_{\overline{p}}^{i}(\mathbb{R}^{n-k}\times(\rc L\backslash \{x\});R).
$
Now, the result is a consequence of \eqref{pro} and \eqref{ConeTW}.

\medskip
The restriction of the product of  blown-up  cochains (see \cite[Proposition 4.2]{CST4}) on the regular part is the 
classical cup product, which is a resolution of the canonical product 
$$\mathbf{\underline{R}}\otimes \mathbf{\underline{R}} \rightarrow \mathbf{\underline{R}}.$$
The result follows from a direct application of \cite[Theorem 1.4]{FR2}, which asserts that this canonical product defined on the regular part of $X$ can be extended uniquely to the Deligne's sheaves. 
\end{proof}

Let us focus on the fact that the product on Deligne sheaves is only defined at the level of the derived category 
$\mathcal{D}(X)$, whereas the blown-up cochains give  a sheaf of perverse algebras. 
This is the starting point of a rational homotopy theoretic treatment of intersection cohomology 
as done in  \cite{CST1}.

\section{Intersection chains and cochains}\label{sec:BMchains}

\begin{quote}
In this section we introduce the intersection chain complex  and its dual, the intersection cochain complex,
together with a Borel-Moore intersection chain complex. 
In \propref{prop:spanier}, we extend to pseudomanifolds  a theorem of Spanier (\cite{Span}) 
originally stated for singular chains and cochains on manifolds.
\end{quote}

\subsection{Borel-Moore intersection chains}

 We denote by
$C^{\BM}_{*}(X;R)$ the complex of formal sums $\xi=\sum_j\lambda_j\sigma_j$, with $\sigma_j$ filtered, 
$\lambda_j\in R$ and such that any $x\in X$ possesses  a neighboorhood  for which the restriction of $\xi$ is finite.
 We say that $\xi$ \textbf{is locally finite or a Borel-Moore chain.} 
 If we restrict to finite linear combinations, we recover the \textbf{filtered chains} and denote by
$C_{*}(X;R)$ the associated complex.

\begin{definition}\label{def:chainBM}
Consider a perverse space  $(X,\ov p)$  and 
 a filtered  simplex
$\sigma\colon\Delta=\Delta_{0}\ast \dots\ast\Delta_{n} \to X$.
\begin{enumerate}[{\rm (i)}]
\item The \textbf{perverse degree of } $\sigma$ is  the $(n+1)$-tuple,
$\|\sigma\|=(\|\sigma\|_0,\dots,\|\sigma\|_n)$,  
where
 $\|\sigma\|_{i}=\dim (\Delta_{0}\ast\dots\ast\Delta_{n-i})$, 
with the convention $\dim \emptyset=-\infty$.
 \item Given a stratum $S$ of  $X$, the \textbf{perverse degree of $\sigma$ along $S$} is defined by \ 
 $$\|\sigma\|_{S}=\left\{
 \begin{array}{cl}
 -\infty,&\text{if } S\cap \im \sigma=\emptyset,\\
 \|\sigma\|_{\codim S}&\text{if } S\cap \im \sigma\ne\emptyset.\\
  \end{array}\right.$$
  \item The filtered singular simplex $\sigma$  is  \textbf{$\ov{p}$-allowable} if
  $
  \|\sigma\|_{S}\leq \dim \Delta-\codim S+\ov{p}(S),
  $
   for any stratum $S$. 
   \item A (finite or locally finite) chain $c$ is 
   \textbf{$\ov{p}$-allowable} if it  is a linear combination of  $\ov{p}$-allowable simplices.
   The chain $c$ is a  \textbf{$\ov{p}$-intersection chain} if $c$ and 
   its boundary $\partial c$ are $\ov{p}$-allowable chains.
   The complexes of $\ov{p}$-intersection chains of $X$  with the differential $\partial$
are denoted by
$C^{\ov{p}}_* (X;R) $ or $C^{\BM,\ov{p}}_* (X;R) $. 
Their respective homology $ H^{\ov{p}}_* (X;R)$  
and
$ H^{\BM,\ov{p}}_* (X;R)$
are respectively the \textbf{$\ov{p}$-intersection homology} 
and \textbf{the Borel-Moore $\ov{p}$-intersection homology} of $X$. 
   \end{enumerate}
   \end{definition}

For a perversity $\ov{p}$ such that $\ov{p}\not\leq \ov{t}$, we may have $\ov{p}$-allowable simplices that
are not regular in the sense of \defref{def:filteredsimp}.
This failure has bad consequences on Poincar\'e duality (see \cite{CST2}).
To overcome this point, the tame intersection homology and cohomology has been introduced in
\cite{MR2210257} and \cite{CST4}. Let us remind it.

\begin{definition}
Given a regular simplex $\Delta = \Delta_0 * \dots *\Delta_n$ we denote by 
$\gd\Delta$ the regular part of the chain $\partial \Delta$.  That is
$\gd \Delta =\partial (\Delta_0 * \dots * \Delta_{n-1})* \Delta_n$, if $|\Delta_n| = 0 $, or $\gd \Delta = \partial \Delta$,
if $|\Delta_n|\geq 1$.
\end{definition}

\begin{definition}\label{tameNormHom}
Let  $(X,\ov p)$ be a perverse space.
 Given a regular simplex $\sigma \colon\Delta \to X$, we define the chain $\gd \sigma$ by $\sigma \circ \gd$.
 Notice that $\gd^2=0$.
 We denote by $\gC_{*} (X;R)$ 
 (resp. $\gC_{*}^{\BM} (X;R)$) the  complex of finite chains 
 (resp. locally finite chains) generated by the  regular simplices, endowed with the differential $\gd$.
 
A $\ov p$-allowable  filtered simplex $\sigma \colon \Delta \to X$  is  \textbf{$\ov{p}$-tame} if $\sigma$ is also a regular simplex. In the two cases, finite or locally finite, a chain $\xi$ is 
   \textbf{$\ov{p}$-tame} if it is a linear combination of  $\ov{p}$-tame simplices.
   A chain $\xi$ is a  \textbf{tame $\ov{p}$-intersection chain} if $\xi$ and $\gd \xi$ are $\ov{p}$-tame chains.

We write  $\gC^{\ov{p}}_* (X;R) \subset \gC_* (X;R)$ the complex of tame $\ov{p}$-intersection finite chains endowed with the differential  $\gd $.
Its homology  $\gH^{\ov{p}}_{*} (X;R)$  is the \textbf{tame $\ov{p}$-intersection homology.} 

In the same manner, we write $\mathfrak{C}_{*}^{\BM,\overline{p}}(X;R)$ for the complex of locally
finite tame $\overline{p}$-intersection chains. Its homology  $\mathfrak{H}_*^{\BM,\overline{p}}(X;R)$ 
is the  \textbf{Borel Moore tame $\ov{p}$-intersection homology}.
\end{definition}

Let us recall the main properties of tame $\ov{p}$-intersection homology (see \cite{CST3} or \cite{FriedmanBook}).
\begin{itemize}
\item It is natural with respect to stratum preserving maps, \cite[Proposition 7.6]{CST3}.
\item It satisfies
Mayer-Vietoris property, \cite[Proposition 7.10]{CST3}. This  comes from  the fact that the classical
subdivision operator $\Sd\colon C_{*}(X;R)\rightarrow C_{*}(X;R)$ can also
be defined on the complex $\mathfrak{C}_{*}^{\overline{p}}(X;R)$
and  gives a quasi-isomorphism. As a consequence if we consider
an open covering $\mathcal{U}$ of $X$ the inclusion of $\mathcal{U}$-small chains
\[
\mathfrak{C}_{*}^{\overline{p},\mathcal{U}}(X;R)\rightarrow\mathfrak{C}_{*}^{\overline{p}}(X;R)
\]
is a quasi-isomorphism, \cite[Corollaire 7.13]{CST3}.
\item We also have a stratified homotopy invariance, \cite[Proposition 7.7]{CST3}. 
\item Let $L$ be a compact filtered space, $\ov{p}$ a perversity on the open cone $\rc L$, of apex $\tv$.
By denoting $\ov{p}$ the perversity induced on $L$, we have
 (\cite[Proposition 7.9]{CST3}): 
\begin{equation}\label{equa:conetame}
\mathfrak{H}_{k}^{\overline{p}}(\overset{\circ}{c}L;R)=
\begin{cases}
\mathfrak{H}_{k}^{\overline{p}}(L;R) & \text{if }\,k\leq D\overline{p}(\tv),\\
0 & \text{if }\, k> D{p}(\tv).
\end{cases}
\end{equation}
\item If $U$ is an open subset of a perverse space $(X,\ov{p})$ with the induced perversity, we define
relative tame $\ov{p}$-intersection chains as the quotient
$\mathfrak{C}_{*}^{\overline{p}}(X,U;R)=
\mathfrak{C}_{*}^{\overline{p}}(X;R)/\mathfrak{C}_{*}^{\overline{p}}(U;R)$. Its homology denoted $\mathfrak H_{i}^{\ov{p}}(X,U ;R)$ fits into an exact sequence,
\begin{equation}\label{suitexacteadroite}
\ldots \to \mathfrak  H_{i}^{\ov{p}}(U ;R) \to  \mathfrak H_{i}^{\ov{p}}(X ;R)  \to \mathfrak H_{i}^{\ov{p}}(X,U ;R)  \to \mathfrak  H_{i-1}^{\ov{p}}(U ;R) \to \ldots
\end{equation}
Also if $K\subset U$ is compact,
we have by excision
$\mathfrak H_{i}^{\ov{p}}(X,X\backslash K ;R)=
\mathfrak H_{i}^{\ov{p}}(U,U\backslash K ;R)$.
\end{itemize}

The next result connects tame $\overline{p}$-intersection chains 
to Borel-Moore ones. 

\begin{proposition}\label{BMCompact}
Let $(X,\overline{p})$ be a perverse space.  Suppose that $X$ is locally compact, metrizable and separable.  
The  complex 
of locally finite tame $\ov{p}$-intersection chains
is isomorphic to the inverse limit of complexes,
\begin{equation}\label{descom}
\xymatrix{
\mathfrak{C}_{*}^{\BM,\overline{p}}(X;R)
\ar[r]^-{\cong}
&
{\varprojlim}_{K\subset X}\mathfrak{C}_{*}^{\overline{p}}(X,X\backslash K;R),
}
\end{equation}
where the limit is taken over all compact subsets of $X$.
\end{proposition}

\begin{proof}
Let $c$ be a chain of $\mathfrak C^{\BM,\overline p}_*(X;R)$.
Following \cite[Proposition A.13]{CST1} we 
know that there exists a locally finite family $\{\sigma\}$ of $\ov p$-intersection simplices, a locally family $\{ \tau \}$ of $\ov p$-bad faces  and a decomposition 
$
c = \sum_{\sigma } r_\sigma \sigma + \sum_{\tau} c_\tau	,
$
where $r_\sigma \in R$ and  
$c_\tau$ is the $\ov p$-intersection chain putting together all the  simplices of $c$ having the same bad face $\tau$.
In particular, any $c \in \mathfrak C^{\BM,\overline p}_*(X;R)$ can be written as  a locally finite decomposition
\begin{equation}\label{desc}
c  = \sum_{k \in J}c_k,
\end{equation}
 where each $c_k$ is a $\ov p$-intersection chain.
Let $K$ be a  compact subset of $X$. We define $J_K = \{ k \in J \mid \supp(c_k) \cap K \ne \emptyset\}$
and a map
$$
\rho_K \colon \mathfrak C^{\BM,\overline p}_*(X;R) \to \mathfrak C^{\overline p}_*(X,X \backslash K;R)
$$
 { by  }
$$
\rho_K(c) = \class_K \left( \sum_{k \in J_K}  c_k\right).
$$
This is a finite sum  since $K$ is compact and $J$ is locally finite. Let us see that $\rho_K$ is well defined. 
If $c = \sum_{k \in J'} e_k$ is another decomposition of $c$ as in  \eqref{desc} then the chain
$$
\sum_{k \in J_K}  c_k - \sum_{k \in J'_K}  e_k = -\sum_{k \not\in J_K}  c_k + \sum_{k \not\in J'_K}  
e_k
$$
does not meet $K$ and therefore 
$
\class_K(\sum_{k \in J_K}  c_k )= {\class}_K ( \sum_{k \in J'_K}  e_k)
$.

Also, we observe that $\rho_K$ clearly commutes with differentials.
By taking the limit over all compact subsets of $X$, we get a chain morphism,
$$\rho\colon \mathfrak{C}_{*}^{\BM,\overline{p}}(X;R)\rightarrow
{\varprojlim}_{K\subset X}\mathfrak{C}_{*}^{\overline{p}}(X,X\backslash K;R).
$$

 If $\rho(c)=0$ then ${\class}_{K}(c)=0$ for any compact set $K$. This implies $c=0$ and therefore the
injectivity of the map $\rho$.

In order to prove the surjectivity, we notice first that the topological hypotheses imply 
(\cite[16C]{Wil}) the existence of an  increasing sequence of compacts
$
K_0 \subset K_1 \subset \cdots K_m \subset \cdots $ such that any compact $K \subset X$ is included in some
$K_m$. In particular, the family $(K_m)$ is cofinal  in the family of compact subsets of $X$ and $X=\bigcup_{m\in \mathbb{N}} K_m$.
Let us consider a collection of chains
\[
\gc =
\left\{\gc_m = \class_K (c^m) \right\}\in{\varprojlim}_{m}\mathfrak{C}_{*}^{\overline{p}}(X,X\backslash K_{m};R).
\]
For $m\geq 0$, the natural inclusion $\gc_{m+1} \mapsto \gc_m$ gives
$\rho_{K_m} (c^{m+1}) = \rho_{K_m}(c^m),
$
that is,
$
\supp(c^{m+1} - c^m) \cap  K_m = \emptyset.
$ 
So, we have
\begin{equation}\label{vacio}
\supp(c^{m+1} - c^m) \cap K_{m'} = \emptyset \hbox{ if } m' \leq m.
\end{equation}
Let us consider the chain
$$
c = \sum_{k\in \mathbb N} (c^{k+1} -c^k).
$$
Condition \eqref{vacio} gives $c \in \mathfrak C^{\BM,\overline p}_*(X;R)$. We end the proof with $\rho(c) = \gc$, that is, 
$$
\class_{K_m} \left( \sum_{k } (c^{k+1} -c^k) \right) = \class_{K_m} (c^m),
$$
for each $m\in \mathbb N$.
Since
$
J_{K_m} = \{ k \in \mathbb N \mid \supp(c^{k+1} - c^k) \cap K_m \ne \emptyset\} \subset \{0, \ldots, m-1\}$ (cf. \eqref{vacio}) then 
$$
\class_{K_m} \left( \sum_{k } (c^{k+1} -c^k) \right)
=
\class_{K_m} \left( \sum_{k < m} (c^{k+1} -c^k) \right) =\class_{K_m} (c^m).
$$
\end{proof}

Propositions \ref{BMtimesR} and \ref{BMCone} 
have been already established in the context of filtration depending perversities
in \cite{MR2276609}.

\begin{proposition}\label{BMtimesR}
Let $L$ be a compact filtered space, $a\in\N$
and $\ov{p}$ a perversity on the product $\R^m\times L$.
By still denoting $\ov{p}$ the perversity induced on $L$, we have
$$\mathfrak{H}_{k}^{\BM,\ov{p}}(\R^a\times L;R)=\mathfrak{H}^{\ov{p}}_{k-a}(L;R).$$
\end{proposition}

\begin{proof}
We consider the following cofinal family of compact subsets of $\R^a\times L$,
$$\left\{K_{n}=[-n,n]^a\times L \mid n\in\N.\right\}.$$
As the open subsets $(\R^a\times L)\backslash K_{n}$ are stratified homeomorphic by the canonical inclusions,
we are reduced to the case $n=0$. From \propref{BMCompact}, we deduce:
\begin{eqnarray*}
\gH_{k}^{\BM,\ov{p}}(\R^a\times L)
&=&
\gH_{k}^{\BM,\ov{p}}(\R^a\times L, (\R^a\times L \backslash\{0\})\times L)
= \oplus_{i+j=k} H_{i}(\R^a, \R^a\backslash \{0\})\otimes \gH_{j}^{\ov{p}}(L)\\
&=&
H_{a}(S^a)\otimes \gH_{k-a}^{\ov{p}}(L)=\gH_{k-a}^{\ov{p}}(L).
\end{eqnarray*}
\end{proof}

Before the next determination of Borel-Moore homology, we need the value of the intersection
homology of a join with an $a$-dimensional  sphere $S^a$. 
Let $X$ be a filtered space.
Recall that the join $S^a\ast X$ is the quotient of the product
with the closed unit ball,
$B^{a+1}\times X$,
by the equivalence relation 
$(z,x)\sim (z,x')$ if $z\in S^a$. 
The strata of $S^a\ast X$ are $S^a$ and the products $\rB^{a+1}\times S$, with $S$ a stratum of $X$
and $\rB^{a+1}$ the open unit ball.
For any perversity on $X$, we get a perversity on the join from $\ov{p}(\rB^{a+1}\times S)=\ov{p}(S)$ and
the choice of a number $\ov{p}(S^a)$. 

\begin{lemma}\label{lem:join}
Let $\ov{p}$ be a perversity on the join $S^a\ast X$. We have
$$\gH_{k}^{\ov{p}}(S^a\ast X;R)=\left\{
\begin{array}{ccl}
\gH_{k}^{\ov{p}}(X;R)&\text{if}&k\leq D\ov{p}(S^a),\\
0&\text{if}&D\ov{p}(S^a)+1\leq k\leq D\ov{p}(S^a)+a+1,\\
\gH_{k-a-1}^{\ov{p}}(X;R)&\text{if}&k\geq D\ov{p}(S^a)+a+2.
\end{array}\right.$$
\end{lemma}

\begin{proof}
Set $B^{a+1}_{1/2}=\{z\in\R^{a+1}\mid \|z\|\leq 1/2\}$
and $F=(B^{a+1}_{1/2}\times X)/\sim$. We apply the Mayer-Vietoris exact
sequence to the  open sets 
$U=(S^a\ast X)\backslash S^a$
and
$V=(S^a\ast X)\backslash F$,
\begin{equation}
\xymatrix@1{
\dots\ar[r]&\gH_{k}^{\ov{p}}(U\cap V)\ar[r]^-{J_{k}}&
\gH_{k}^{\ov{p}}(U)\oplus \gH_{k}^{\ov{p}}(V)\ar[r]&
\gH_{k}^{\ov{p}}(S^a\ast X)\ar[r]&
\gH_{k-1}^{\ov{p}}(U\cap V)\ar[r]^-{J_{k-1}}&\dots
}
\end{equation}
and get a short exact sequence
\begin{equation}\label{equa:MVraccourci}
\xymatrix@1{
0\ar[r]&
\Coker\, J_{k}\ar[r]&
\gH_{k}^{\ov{p}}(S^a\ast X)\ar[r]&
\Ker\, J_{k-1}\ar[r]&
0.
}
\end{equation}
We study the map $J_{\ell}$ and for that determine first,
\begin{itemize}
\item $\gH_{\ell}^{\ov{p}}(U)=\gH_{\ell}^{\ov{p}}(\rB^{a+1}\times X)=\gH_{\ell}^{\ov{p}}(X)$,
\item $\gH_{\ell}^{\ov{p}}(V)=\gH_{\ell}^{\ov{p}}((S^a\times ]1/2,1]\times X)/\sim)=
\gH_{\ell}^{\ov{p}}(S^a\times \rc X)=\oplus_{i+j=\ell}H_{i}(S^a)\otimes \gH_{j}^{\ov{p}}(\rc X)$,
\item $\gH_{\ell}^{\ov{p}}(U\cap V)=\gH_{\ell}^{\ov{p}}(S^a\times ]1/2,1[\times X)=
\gH_{\ell}^{\ov{p}}(S^a\times X)=\oplus_{i+j=\ell}H_{i}(S^a)\otimes \gH_{j}^{\ov{p}}( X)$.
\end{itemize}
The map $J_{\ell}$ can be written as
$$\oplus_{i+j=\ell}H_{i}(S^a)\otimes \gH_{j}^{\ov{p}}( X)\to
\gH_{\ell}^{\ov{p}}(X)\oplus \oplus_{i+j=\ell}H_{i}(S^a)\otimes \gH_{j}^{\ov{p}}( \rc X)$$
which can be simplified in
$$ \gH_{\ell-a}^{\ov{p}}( X)\to
\gH_{\ell}^{\ov{p}}(\rc X) \oplus \gH_{\ell-a}^{\ov{p}}(\rc X).$$
From (\ref{equa:conetame}), we deduce
\begin{eqnarray*}
\Ker \,J_{\ell}&=&
\left\{\begin{array}{ccl}
0&\text{if}&\ell\leq D\ov{p}(S^a)+a,\\
\gH_{\ell-a}^{\ov{p}}(X)&\text{if}&\ell>D\ov{p}(S^a)+a,
\end{array}\right.\\
\Coker \,J_{\ell}
&=& \gH_{\ell}^{\ov{p}}(\rc X).
\end{eqnarray*}
The result follows from (\ref{equa:conetame}) and (\ref{equa:MVraccourci}).
\end{proof}

\begin{proposition}\label{BMCone}
Let $L$ be a compact filtered space, $a\in\N$
and $\ov{p}$ a perversity on the open cone $\rc L$, of apex $\tv$.
By denoting $\ov{p}$ the perversity induced on $L$, we have
\[
\mathfrak{H}_{k}^{\BM,\overline{p}}(\mathbb{R}^a\times \rc L;R)=
\begin{cases}
0& \text{if } k\leq a+D\overline{p}(\tv)+1,\\
\tilde{\mathfrak{H}}_{k-a-1}^{\overline{p}}(L;R)  & \text{if } k>a+ D\overline{p}(\tv)+1.
\end{cases}
\]
\end{proposition}

\begin{proof}
We consider the following cofinal family of compact subsets of $\R^a\times \rc L$,
$$\left\{K_{n}=[-n,n]^a\times \tc_{n} L \mid n\in\N^*\right\}\cup  \left\{K_{0}=\{(0,\tv)\}\right\},$$
where $\tc_{n}L=\{[x,t]\in \rc L\mid t\leq n/(n+1)\}$.
As the open subsets $(\R^a\times \rc L)\backslash K_{n}$ are stratified homeomorphic by the canonical inclusions,
we are reduced to the case $n=0$. 
From \propref{BMCompact}, we deduce:
$$\mathfrak{H}_{k}^{\BM,\overline{p}}(\mathbb{R}^a\times\overset{\circ}{c}L)=
\gH_{k}^{\ov{p}}(\R^a\times \rc L, (\R^a\times \rc L)\backslash \{(0,\tv)\}).$$
Let $\tu$ and $\tv$ be the respective apex of 
$ \rc(S^{a-1}\ast L)$ and $\rc L$.
Let $\alpha\colon [0,1[\to [0,\infty[$ be a fixed homeomorphism.
From \cite[5.7.4]{MR2273730}, we get the homeomorphism
$$h\colon \rc(S^{a-1}\ast L)\backslash \{\tu\}
\to
(\R^a\times \rc L)\backslash \{(0,\tv)\},
$$
given by
\begin{equation}
h([z,y],r)=\left\{
\begin{array}{ccl}
(2rz,[y,\alpha(r)])&\text{if}&\|z\|\leq 1/2,\\
(rz/\|z\|,[y,\alpha(2r(1-\|z\|))])&\text{if}&\|z\|\geq 1/2.
\end{array}\right.
\end{equation}
Let us see that $h$ preserves the stratifications. We distinguish the following  cases.
\begin{itemize}
\item $h(S^{a-1}\times ]0,1[)=(\R^a\times\{\tv\})\backslash \{(0,\tv)\}$ since
$h([z,y],r)=(rz,\tv)$ if $\|z\|=1$.
\item The restriction $h\colon \rB^a\times L\times ]0,1[\to \R^a\times L\times ]0,1[$,
given by
$$(z,y,r)\mapsto\left\{
\begin{array}{ccl}
(2rz,y,\alpha(r))&\text{if}&\|z\|\leq 1/2,\\
(rz/\|z\|,y,\alpha(2r(1-\|z\|)))&\text{if}&\|z\|\geq 1/2,
\end{array}\right.$$
\end{itemize}
is clearly a stratified homeomoprhism. 
The determination of the inclusion map
$$
\gH_{k}^{\ov{p}}((\R^a\times \rc L)\backslash \{(0,\tv)\})\to
\gH_{k}^{\ov{p}}(\R^a\times \rc L)
$$
can thus be replaced by the determination of
$$\gH_{k}^{\ov{p}}(\rc (S^{a-1}\ast L)\backslash \{\tu\})\to
\gH_{k}^{\ov{p}}(\rc L).$$
Using \lemref{lem:join} for the domain and (\ref{equa:conetame}) for the range, we obtain,
\begin{equation}
\left\{\begin{array}{ccl}
\gH_{k}^{\ov{p}}(L)\xrightarrow{\id}\gH_{k}^{\ov{p}}(L)&\text{if}&k\leq D\ov{p}(\tv),\\
0\to 0&\text{if}&D\ov{p}(\tv)+1\leq k\leq D\ov{p}(\tv)+a,\\
\gH_{k-a}^{\ov{p}}(L)\to 0&\text{if}&k\geq D\ov{p}(\tv)+a+1.
\end{array}\right.
\end{equation}
The result follows from the long exact sequence of intersection homology.
\end{proof}

Mention from \cite[Proposition 6.7]{CST4} the existence of cap products,
\begin{equation}\label{equa:caphomology}
 \widetilde{N}_{\overline{q}}^{i}(X;R)\otimes \mathfrak{C}_{j}^{\overline{p}}(X;R)
\xrightarrow{-\cap -} \mathfrak{C}_{j-i}^{\overline{p}+\overline{q}}(X;R)
\end{equation}
such that (\cite[Proposition 6.6]{CST4})
$$
(\eta\cup \omega)\cap \xi=\eta\cap(\omega\cap\xi).
$$
Hence, the collection $\{\mathfrak{C}_{*}^{\overline{p}}(X;R)\}_{\overline{p}\in\mathcal{P}}$
is a left perverse module over the perverse algebra $\left\{ \widetilde{N}_{\overline{q}}^{*}(X;R)\right\} _{\overline{q}\in\mathcal{P}}$.
This cap product 
can also be extended in a cap product,
\begin{equation}\label{equa:capBM}
\widetilde{N}_{\overline{q}}^{i}(X;R)\otimes \mathfrak{C}_{j}^{\BM,\overline{p}}(X;R)
\xrightarrow{-\cap -} \mathfrak{C}_{j-i}^{\BM,\overline{p}+\overline{q}}(X;R),
\end{equation}
as follows. As (see (\ref{desc}))  any $c \in \mathfrak C^{\BM,\overline p}_*(X;R)$ can be written as  a locally finite sum of 
 $\ov p$-intersection chains,
$ \xi  = \sum_{k \in J}\xi_k$, we set
$$\omega\cap \xi= \sum_{k \in J} (\omega\cap \xi_{k}).$$
A verification shows that this cap product is well defined and makes of
$\{\mathfrak{C}_{*}^{\BM,\overline{p}}(X;R)\}_{\overline{p}\in\mathcal{P}}$
 a left perverse module over the perverse algebra $\left\{ \widetilde{N}_{\overline{q}}^{*}(X;R)\right\} _{\overline{q}\in\mathcal{P}}$. 

\subsection{Intersection cochains}

The tame $\ov{p}$-intersection cohomology $\mathfrak{H}^*_{\ov{p}}(X;R)$
is defined from the dual complex
$\mathfrak{C}^*_{\ov{p}}(X;R)=\Hom(\mathfrak{C}_{*}^{\ov{p}}(X;R),R)$.
We still denote $\gd$ its differential.
This cohomology satisfies properties  as those already listed for the tame homology (see \cite{CST4}):
naturality with respect to stratum preserving maps,  Mayer-Vietoris property and  invariance for stratified homotopy.
Moreover, the case of an open cone on a compact perverse space is determined by 
\[
\mathfrak{H}_{\overline{p}}^{k}(\overset{\circ}{c}L;R)=
\begin{cases}
\mathfrak{H}_{\overline{p}}^{k}(L;R) & \text{if } k\leq D{p}(\tv),\\
\Ext(\mathfrak{H}_{k-1}^{\overline{p}}(L;R),R) & \text{if } k=D\overline{p}(\tv)+1,\\
0 & \text{if } k>D\overline{p}(\tv)+1.
\end{cases}
\]
The cap product (\ref{equa:caphomology}) defines a map
\begin{equation}\label{equa:capcohomology}
\star\colon
\mathfrak{C}^{j}_{\overline{p}}(X;R)
\otimes
 \widetilde{N}_{\overline{q}}^{i}(X;R)
 \longrightarrow
 \mathfrak{C}^{j+i}_{\overline{p}-\overline{q}}(X;R)
\end{equation}
by
$$(c\star \omega)(\xi)=c(\omega\cap\xi).$$
We check easily
$c\star (\omega\cup \eta)=(c\star\omega)\star\eta$.
Hence, the collection $\{\mathfrak{C}^{*}_{\overline{p}}(X;R)\}_{\overline{p}\in\mathcal{P}}$
is a right perverse module over the perverse algebra 
$\left\{ \widetilde{N}_{\overline{q}}^{*}(X;R)\right\} _{\overline{q}\in\mathcal{P}}$. 
This operation respects the cochains with compact supports and makes
$\{\mathfrak{C}^{*}_{\overline{p},c}(X;R)
\}_{\ov{p}\in\cP}$
a right perverse module over the perverse algebra
$\{\tN_{\ov{q}}^*(X;R)\}_{\ov{q}\in\cP}$. 

\medskip
There is also a notion of $\mathcal{U}$-small cochains relatively to an open covering, as follows.

\begin{definition}
Let $\mathcal{U}$ be  an open covering  of $X$ and $\ov{p}$ a perversity on $X$. 
The complex
of \textbf{$\mathcal{U}$-small cochains} is  $\mathfrak{C}_{\overline{p}}^{*,\mathcal{U}}(X;R)
= \Hom (\mathfrak{C}^{\overline{p},\mathcal{U}}_{*}(X;R),R)$.
If $V$ is an open
subset of $X$, we recall the category $\Cov(V)$ of \defref{def:covV}.
The \textbf{sheaf of intersection cochains} $\mathbf{C}_{\overline{p}}^{*}$
is defined by its sections as the direct limit 
\begin{equation}\label{equa:SoftfaltC}
\Gamma(V;\mathbf{C}_{\overline{p}}^{*})=
\colim_{\mathcal{U}\in \Cov(V)}\mathfrak{C}_{\overline{p}}^{*,\mathcal{U}}(V;R).
\end{equation}
\end{definition}

\begin{proposition}\label{SoftFlatC}
Let $(X,\ov{p})$ be a perverse space. The complex of sheaves 
$\mathbf{C}^*_{\overline{p}}$ over $X$
is a complex of soft and flat sheaves.
\end{proposition}

\begin{proof}
(Softness) We have defined, on the complex of cochains $\mathfrak{C}_{\overline{p}}^{*}(X;R)$, a structure of  right module over $\widetilde{N}^*_{\overline{0}}(X;R)$. This structure  being compatible with the restriction maps,
  the sheaf  $\mathbf{C}_{\overline{p}}^{*}$ is a sheaf of right modules over $\widetilde{\mathbf{N}}_{\overline{0}}^{*}$. 
  Thus the softness of $\widetilde{\mathbf{N}}_{\overline{0}}^{*}$ implies the softness of $\mathbf{C}_{\overline{p}}^{*}$.  
\\
(Flatness) The proof is similar to the proof made for $\widetilde{\mathbf{N}}_{\overline{p}}^{*}$ (see \ref{SoftFlatTW}).
\end{proof}

Mention also the existence in \cite{FM} of an intersection cohomology with compact supports defined by
$\gH^*_{\ov{q},c}(X;R)=\lim_{K \text{\rm compact}}\gH^*_{\ov{q}}(X,X\backslash K;R)$.
In \cite{CST2}, \cite{CST4} we have connected these cohomologies with the blown-up cohomology.
In particular, we have established the following isomorphisms.

\begin{proposition}\label{prop:TWadroite}
Let $(X,\ov{p})$  be a paracompact perverse pseudomanifold and $R$ a field. 
Then there are  isomorphisms
\begin{equation}\label{equa:GMdroiteTWc}
\crH^*_{\ov{p},c}(X;R)\cong \gH^*_{D\ov{p},c}(X;R) 
\text{ and }
\crH^*_{\ov{p}}(X;R)\cong \gH^*_{D\ov{p}}(X;R) 
\end{equation}
\end{proposition}


\subsection{Bidual}

Let us take  an injective resolution of $R$,
$$0\rightarrow R\rightarrow Q\stackrel{\rho}{\twoheadrightarrow} Q/R\rightarrow 0.$$
We denote by $I_R^*$ the cochain complex 
$\xymatrix@1{
I_R^0=Q\ar@{->>}[r]^-{\rho}
&
I^1_R=Q/R
}$, of differential $d=\rho$.
We define a chain complex 
${\Hom}_{*}(\mathfrak{C}^*_{\overline{p},c}(X;R),I_R^*)
$,
whose  set
${\Hom}_{k}(\mathfrak{C}^*_{\overline{p},c}(X;R),I_R^*)
$
of elements of degree $k$ is formed of the couples $(\varphi_{0},\varphi_{1})$ such that
$$\varphi_{0}\colon \mathfrak{C}^k_{\overline{p},c}(X;R) 
\to Q\text{ and } 
\varphi_{1}\colon \mathfrak{C}^{k+1}_{\overline{p},c}(X;R)  \to Q/R,$$
and whose  differential is defined by
$\mathbf{d}(\varphi_{0},\varphi_{1})= ((-1)^k\varphi_{0}\circ \gd, (-1)^k\varphi_{1}\circ \gd + d\circ \varphi_{0})
$.

\smallskip
 The star product (\ref{equa:capcohomology}) defines a map
 \begin{equation}\label{equa:dualstar}
 \perp\colon \tN^i_{\ov{q}}(X;R)\otimes \Hom_{j}(\mathfrak{C}^*_{\ov{p},c},I_{R}^*)
 \to
 \Hom_{j-i}(\mathfrak{C}^*_{\ov{p}+\ov{q},c},I_{R}^*)
 \end{equation}
by
$$(\omega\perp\varphi)(c)=\varphi(c\star\omega).$$
We check easily
$(\omega\cup \eta)\perp \varphi=\omega\perp (\eta \perp \varphi)$.
Hence, the collection $\left\{\Hom_{*}(\mathfrak{C}^*_{\ov{p},c},I_{R}^*)\right\}_{\ov{p}\in\cP}$
is a left perverse module over the perverse algebra 
$\left\{ \widetilde{N}_{\overline{q}}^{*}(X;R)\right\} _{\overline{q}\in\mathcal{P}}$.

\smallskip
The following result is the well-known ``biduality'' for intersection homology, proved in 
\cite[Theorem V.8.10]{Bor} in  sheaf theory 
 and in classical algebraic homology in \cite{Span}, see also \cite{MR1366538}. 
We give a short proof below.

\begin{proposition} \label{prop:spanier}
Let $(X,\ov{p})$ be a paracompact perverse pseudomanifold. 
Suppose that $X$ admits an exhaustive sequence of compacts $X=\cup_{i}K_{i}$, $K_{i}\subset \mathring{K}_{i+1}$,
such that the $R$-module 
\begin{equation}\label{equa:finite}
\mathfrak{H}^*_{\overline{p}}(X,X\backslash K_{i};R) \text{ is finitely generated.}
\end{equation}
Then the morphism
\begin{equation}\label{equa:cochainchain}
\Phi_X\colon \mathfrak{C}^{\BM,\overline{p}}_{k}(X;R)\rightarrow 
{\Hom}_{k}(\mathfrak{C}^*_{\overline{p},c}(X;R),I_R^*),
\end{equation}
defined by $\Phi_{X}(\xi)=(\varphi(\xi)_{0},0)$ with
$\varphi(\xi)_{0}(c)=c(\xi)\in R\subset Q$, is well defined and a quasi-isomorphism,
linear for the  structure of 
perverse modules over the perverse algebra 
$\left\{ \widetilde{N}_{\overline{q}}^{*}(X;R)\right\} _{\overline{q}\in\mathcal{P}}$. 
\end{proposition}

\begin{remark}
The hypothesis (\ref{equa:finite}) is satisfied for \emph{PL-pseudomanifolds.} This comes from
the existence of isomorphisms between
\begin{itemize}
\item the intersection homology and its PL-version, see \cite[Theorem 6.3.31]{FriedmanBook},
\item the PL-intersection homology and the simplicial intersection homology associated to a triangulation 
of $X$, \cite[Theorem 6.3.30]{FriedmanBook}.
\end{itemize}
\end{remark}

\begin{proof}[Proof of \propref{prop:spanier}]
From 
$\mathfrak{C}^{*}_{\overline{p},c}(X;R)\cong 
{\varinjlim}_{K_{i}}\mathfrak{C}_{*}^{\overline{p}}(X,X\backslash K_{i};R)
$
and \propref{BMCompact}, we may replace (\ref{equa:cochainchain}) by the map
$$\psi_{X}\colon
{\varprojlim}_{K_{i}}\mathfrak{C}_{k}^{\overline{p}}(X,X\backslash K_{i};R)
\to
{\varprojlim}_{K_{i}}\Hom_{k}(\mathfrak{C}_{*}^{\overline{p}}(X,X\backslash K_{i};R),I_{R}^*)
,$$
defined as the  limit of the maps
$$\psi_{X,i}\colon 
\mathfrak{C}_{k}^{\overline{p}}(X,X\backslash K_{i};R)
\to
\Hom_{k}(\mathfrak{C}_{*}^{\overline{p}}(X,X\backslash K_{i};R),I_{R}^*)$$
with
$\psi_{X,i}(\xi)(c)=c(\xi)\in R$. Observe that the natural maps
$$\Hom_{k}(\mathfrak{C}_{*}^{\overline{p}}(X,X\backslash K_{i+1};R),I_{R}^*)
\to
\Hom_{k}(\mathfrak{C}_{*}^{\overline{p}}(X,X\backslash K_{i};R),I_{R}^*)$$
and
$$\mathfrak{C}_{k}^{\overline{p}}(X,X\backslash K_{i+1};R)
\to
\mathfrak{C}_{k}^{\overline{p}}(X,X\backslash K_{i};R)
$$
are surjective. The Mittag-Leffler condition being satisfied,  we have Milnor exact sequences
(\cite[Page 605]{MR516508}). Therefore, 
the map $\psi_{X}$ is a quasi-isomorphism if the
$\psi_{X,i}$'s are so. 
With the finiteness hypothesis on
$\mathfrak{H}^*_{\overline{p}}(X,X\backslash K_{i};R)$
this last property is a  classical result in homological algebra, see
\cite[Proposition 1.3]{MR1366538} for instance.
\end{proof}

\section{Poincar\'e and Verdier dualities}\label{sec:verdier}

\begin{quote}
In  the case of a PL-pseudomanifold, we construct a commutative diagram
of quasi-isomorphisms,
\begin{equation}\label{equa:triangle}\xymatrix{
\widetilde{N}^k_{\overline{p}}(X;R)
\ar[r]^-{\mathcal{D}_{X}}
\ar[rd]_-{DP_X}
&{\Hom}_{n-k}(\mathfrak{C}^{*}_{\overline{p},c}(X;R),I_R^*)
\\
&
\mathfrak{C}_{n-k}^{\BM,\overline{p}}(X;R),
\ar[u]_-{\Phi_X}
}
\end{equation}
where $DP_{X}$ is the Poincar\'e morphism, 
defined by the cap product with a fundamental cycle and
$\mathcal{D}_{X}$ is related to Verdier duality.
\end{quote}

We first bring back  definitions and basic properties of the elements of this diagram.
Recall that $R$ is a principal ideal domain together with an injective resolution  $I_{R}^*$.
%
We refer to  (\cite[chapter V]{Bor},  \cite{Bre}, \cite{Ive}) for more details on Verdier duality. 
For the construction of the quoted maps, we proceed along the next steps.
\begin{itemize}
\item In \propref{prop:VerdierDual}, we show that the Verdier dual sheaf 
$\mathbb{D} (\mathbf{C}^*_{\overline{p}}[n])$ is 
quasi-isomorphic to the Deligne's sheaf. %
\item If a  quasi-isomorphism 
$\Psi_{\overline{p}}\colon \widetilde{\mathbf{N}}^*_{\overline{p}}\rightarrow \mathbb{D} 
(\mathbf{C}^*_{\overline{p}}[n])$ is given, we show  in \propref{prop:explicitqiso}
how one can pertube it in a $ \widetilde{\mathbf{N}}^*_{\ov{\bullet}}$-linear map  and keep a quasi-isomorphism.
\item  Finally,  in \propref{prop:module}, we construct the two quasi-isomorphisms $\cD_{X}$ and $DP_{X}$ 
making the diagram (\ref{equa:triangle}) commutative.
\end{itemize}
This program is based on the existence of structures of perverse modules on the various complexes, see
(\ref{equa:dualstar}) and (\ref{equa:capBM}). Therefore any linear map (for these structures) of domain
$\tN^*_{\overline{p}}(X;R)$ is entirely determined by the image of the constant map $1_{X}$.

\begin{definition} 
Let $X$ be an oriented PL-pseudomanifold of dimension $n$
and
$\mathbf{C}^*_{\ov{p}}$ the sheaf of $\ov{p}$-intersection cochains defined in (\ref{equa:SoftfaltC}).
The {\bf Borel-Moore-Verdier dual} of $\mathbf{C}^*_{\ov{p}}$ 
is the complex of sheaves
$\mathbb{D} (\mathbf{C}^*_{\overline{p}}[n])$ 
defined by
$$U\mapsto {\Hom}_{n-k}(\Gamma_c(U;\mathbf{C}^*_{\ov{p}}),I_R^*).$$ 
This is a cochain complex of sheaves.
\end{definition}

\begin{proposition}\label{prop:VerdierDual}
Let $X$ be an oriented PL-pseudomanifold of dimension $n$. 
The complex of sheaves $\mathbb{D} (\mathbf{C}^*_{\overline{p}}[n])$
is quasi-isomorphic to the Deligne's sheaf.
\end{proposition}

\begin{proof}
Let us prove that the complex of sheaves $\mathbb{D} (\mathbf{C}^*_{\overline{p}}[n])$ is quasi-isomorphic to the
 Deligne sheaf $\mathbf{Q}^*_{\overline{p}}$. As the sheaf $\mathbf{C}^*_{\overline{p}}$ is flat and soft 
 (\propref{SoftFlatC}), the sheaf $\mathbb{D} (\mathbf{C}^*_{\overline{p}}[n])$ is an injective sheaf, hence it is soft. 
 We are reduced to verify that the conditions of \defref{def:deligne} are fulfilled.

From \propref{prop:spanier}, we may replace 
${\Hom}_{n-k}(\mathfrak{C}^*_{\overline{p},c}(X;R),I_R^*)$ by 
$\mathfrak{C}^{\BM,\overline{p}}_{n-k}(X;R)$.
Now, if $U\cong \R^n$, we have $\oplus_{k}H^{\BM}_{n-k}(U)=H^0(U)$
and Property (1) is satisfied.

As in the proof of \thmref{DeligneTW}, Property (2) is a consequence of 
$\gH_{n-i}^{\BM,\ov{p}}(\R^{n-k}\times \rc L)=0$ if $i>\ov{p}(S)$. This last property has been 
established in \propref{BMCone}.

The third condition is equivalent to the fact that the map
$\gH_{n-i}^{\BM,\ov{p}}(\R^{n-k}\times \rc L)\to
\gH_{n-i}^{\BM,\ov{p}}(\R^{n-k}\times (\rc L\backslash \{\tv\})$
is a quasi-isomorphism for $i\leq \ov{p}(S)$. This is a consequence of Propositions \ref{BMtimesR} and 
\ref{BMCone}.
\end{proof}

The sheaves $\mathbb{D} (\mathbf{C}^*_{\overline{p}}[n])$ and $\widetilde{\mathbf{N}}^*_{\overline{p}}$ are both quasi-isomorphic to Deligne's sheaves. As the sheaf 
$\mathbb{D} (\mathbf{C}^*_{\overline{p}}[n])$ is injective there exists for any perversity $\overline{p}$ a (non unique) quasi-isomorphism of sheaves 
$\boldsymbol{\chi}_{\overline{p}}\colon \widetilde{\mathbf{N}}^*_{\overline{p}}\rightarrow \mathbb{D} (\mathbf{C}^*_{\overline{p}}[n])$. 
In the next proposition, we pertube it in a quasi-isomorphism compatible with the structures of perverse
modules.

\begin{proposition}\label{prop:explicitqiso}
Let $X$ be be an oriented PL-pseudomanifold of dimension $n$ and $U\subset X$ an open subset.
Let $\beta_{U}\colon \Gamma_{c}(U,\mathbb{D} (\mathbf{C}^*_{\overline{0}}[n])\to I_{R}^*$ be a cocycle
in the class of $\chi_{\ov{0}}(1_{U})$. 
For any perversity $\ov{p}$ on $X$, we construct a morphism of complex of sheaves 
$$\boldsymbol{\Psi}_{\overline{p}}\colon \widetilde{\mathbf{N}}^*_{\overline{p}}
\rightarrow 
\mathbb{D}(\mathbf{C}^*_{\overline{p}}[n])$$ 
at the level of sections,
$\Psi_{\ov{p}}\colon \Gamma(U;\widetilde{\mathbf{N}}^*_{\overline{p}})\to \Gamma_{c}(U,\mathbb{D} (\mathbf{C}^*_{\overline{0}}[n])
$, by
$$\Psi_{\ov{p},U}(\omega)=\omega\perp \beta_{U}.$$
Then $\boldsymbol{\Psi}_{\ov{p}}$ is a quasi-isomorphism, linear for the structure of perverse modules 
over the perverse algebra 
$\left\{ \widetilde{N}_{\overline{q}}^{*}(X;R)\right\} _{\overline{q}\in\mathcal{P}}$. 
\end{proposition}

\begin{proof}
By properties of Deligne sheaves it suffices to check if the restriction to the regular part $X_{\reg}$ of $X$  is a quasi-isomorphism.  Moreover, when restricted to the regular part, for any perversity $\overline{p}$ we have the following equality of complexes of sheaves 
$$(\widetilde{\mathbf{N}}^*_{\overline{p}})_{\vert X_{\reg}}= (\widetilde{\mathbf{N}}^*_{\overline{0}})_{\vert X_{\reg}}
\text{ and }
(\mathbb{D} (\mathbf{C}^*_{\overline{p}}[n]))_{\vert X_{\reg}}=
(\mathbb{D} (\mathbf{C}^*_{\overline{0}}[n]))_{\vert X_{\reg}}.
$$ 
Thus it is sufficient to consider the perversity $\ov{0}$.

Let $U\cong \R^n$ be an open subset of $X_{\reg}$. 
As $\Psi_{\ov{0},U}(1_{U})=1_{U}\bot \beta_{U}=\beta_{U}$ and $\beta_{U}$ is in the class of 
$\chi_{\ov{0}}(1_{U})$, the  difference $\Psi_{\ov{0},U}(1_{U})-\chi_{\ov{0}}(1_{U})$ is a boundary. 
The map 
 $\chi_{\ov{0}}$ being a quasi-isomorphism, the generator 
$[1_{U}]$ of $H^0(U;\widetilde{\mathbf{N}}^*_{\overline{0}})$
is sent on the generator $[\Psi_{\ov{0},U}(1_{U})]=[\chi_{\ov{0}}(1_{U})]$
of $H^0(U;\mathbb{D}(\mathbf{C}^*_{\overline{0}}[n]))$.
The result follows.
\end{proof}

\emph{For the rest of this section, the space $X$ is an oriented, separable, metrizable PL-pseudomanifold 
of dimension $n$ and $\ov{p}$ is a perversity.}
Since the map $\Phi_{X}\colon \gC^{\BM,\ov{0}}_{*}(X;R)\to \Hom_{n-*}(\gC^*_{\ov{0},c}(X;R),I_{R}^*)$
is a quasi-isomorphism, there exists a cocycle $\gamma_{X}\in \gC_{n}^{\BM,\ov{0}}(X;R)$
such that $\Phi_{X}(\gamma_{X})-\chi_{\ov{0}}(1_{X})$ is a boundary.
The idea is to use \propref{prop:explicitqiso} with $U=X$ and $\beta_{U}=\Phi_{X}(\gamma_{X})$. More precisely, we
define the two maps $\cD_{X}$ and $DP_{X}$ by
\begin{equation}\label{equa:enfin}
\cD_{X}(\omega)=\omega\perp \Phi_{X}(\gamma_{X})
\text{ and }
DP_{X}(\omega)=\omega\cap \gamma_{X}.
\end{equation}

\begin{proposition}\label{prop:module}
The two maps $\cD_{X}$ and $DP_{X}$ make commutative the diagram (\ref{equa:triangle}) and are quasi-isomorphisms.
\end{proposition}

\begin{proof}
We identify the homology of the sections $\Gamma(X;\widetilde{\mathbf{N}}^*_{\overline{p}})$
with the homology of $\widetilde{N}^*_{\overline{p}}(X;R)$,
and similarly for the homologies of $\Gamma(X;\mathbb{D}(\mathbf{C}^*_{\overline{p}}[n]))$ and 
${\Hom}_{n-*}(\mathfrak{C}^{*}_{\overline{p},c}(X;R),I_R^*)$.

 By construction we have $\cD_{X}(1_{X})=\Phi_{X}(DP_{X}(1_{X}))$.
As the three maps $DP_{X}$, $\cD_{X}$ and $\Phi_{X}$ are compatible with the structure of perverse modules, 
we obtain the commutativiy of (\ref{equa:triangle}).

The maps $\Phi_{X}$ (see \propref{prop:spanier}) and $\cD_{X}$ (see \propref{prop:explicitqiso}) are quasi-isomor\-phisms. 
Thus $DP_{X}$ is one also.
\end{proof}

The next statement is a direct consequence of \propref{prop:module}.

\begin{corollary}\label{cor:poincareBM}
The cap product with the fundamental class $[\gamma_X]$ induces an isomorphism
between the blown-up intersection cohomology and the  locally finite tame intersection homology,
$$
\crH^k_{\ov{p}}(X;R)\cong
\mathfrak{H}_{n-k}^{\BM,\overline{p}}(X;R)
.$$
\end{corollary}

From this corollary and \cite[Theorem B]{CST2}, we deduce the existence of a commutative diagram 
between the two duality maps, corresponding to the blown-up cohomology 
and the blown-up cohomology with compact supports.

\begin{corollary}
Let  $\gamma_X\in \mathfrak{C}^{\BM,\overline{0}}_n(X;R)$  be a cycle representing the fundamental class.
 Then there is a commutative diagram whose 
vertical arrows are quasi-isomorphisms,
$$\xymatrix{
\widetilde{N}_{c,\overline{p}}^*(X;R)
\ar@{^(->}[r] 
\ar[d]_-{-\cap \gamma_{X}}
&
\widetilde{N}^*_{\overline{p}}(X;R)
\ar[d]^-{-\cap \gamma_{X}}
\\
\mathfrak{C}^{\overline{p}}_{n-*}(X;R)
\ar@{^(->}[r] 
&
\mathfrak{C}^{\BM,\overline{p}}_{n-*}(X;R).
}$$
\end{corollary}


\providecommand{\bysame}{\leavevmode\hbox to3em{\hrulefill}\thinspace}
\providecommand{\MR}{\relax\ifhmode\unskip\space\fi MR }
\providecommand{\MRhref}[2]{%
  \href{http://www.ams.org/mathscinet-getitem?mr=#1}{#2}
}
\providecommand{\href}[2]{#2}

\end{document}